\documentclass[a4paper]{article}

\usepackage[utf8]{inputenc}
\usepackage[T1]{fontenc}

\usepackage{amsmath, amsthm,multirow}
\usepackage{amssymb}
\usepackage{hyperref}
\usepackage{cite}
\usepackage{cuted}
\usepackage{mathtools}

\newtheorem{defi}{Definition}[section]
\newtheorem{theorem}[defi]{Theorem}
\newtheorem{prop}[defi]{Proposition}
\newtheorem{lemma}[defi]{Lemma}
\newtheorem{cor}[defi]{Corollary}
\newtheorem{remark}[defi]{Remark}
\theoremstyle{definition}
\newtheorem{assumption}[defi]{Assumption}

\newtheorem{ex}[defi]{Example}
\newtheorem{alg}[defi]{Algorithm}

\newcommand{\R}{\mathbb{R}}
\newcommand{\C}{\mathbb{C}}
\newcommand{\N}{\mathbb{N}}
\newcommand{\F}{\mathbb{F}}
\DeclareMathOperator{\rk}{rk}
\DeclareMathOperator{\im}{im}
\DeclareMathOperator{\dom}{dom}
\DeclareMathOperator{\Span}{span}
\DeclareMathOperator{\Sp}{sp}
\DeclareMathOperator{\adj}{adj}

\title{Feedback rectifiable pairs and stabilization of switched linear systems
 \thanks{\textbf{Acknowledgments:} HG is partially supported by the BMBF project EIZ-Project number 03SF0693A.}
        }

\author{M.C.\ Honecker \thanks{Control Systems Group, TU Ilmenau, Helmholtzplatz 5, 98693 Ilmenau, Germany} \and H.\ Gernandt \thanks{Fraunhofer IEG, Fraunhofer Research Institution for Energy Infrastructures and Geothermal Systems IEG, Cottbus, Gulbener Straße 23, 03046 Cottbus, Germany and Institute for Mathematical Modelling, Analysis and Computational Mathematics, Bergische~Universität~Wuppertal, Gau\ss stra\ss e~20, 42119 Wuppertal, Germany} \and K.\ Wulff \thanks{Control Systems Group, TU Ilmenau, Helmholtzplatz 5, 98693 Ilmenau, Germany, \texttt{kai.wulff@tu-ilmenau.de} (Corresponding author)}
\and C.\ Trunk \thanks{Applied Functional Analysis Group, TU Ilmenau, Weimarer Straße 25, 98693 Ilmenau, Germany} \and J.\ Reger \thanks{Control Systems Group, TU Ilmenau, Helmholtzplatz 5, 98693 Ilmenau, Germany}}

\begin{document}

\maketitle

\begin{abstract}
We address the feedback design problem for switched linear systems.
In particular we aim to design a switched state-feedback such that the resulting closed-loop subsystems share the same eigenstructure.
To this effect we formulate and analyse the feedback rectification problem for pairs of matrices.
We present necessary and sufficient conditions for the feedback rectifiability of pairs for two subsystems and give a constructive procedure to design stabilizing state-feedback for a class of switched systems. In particular the proposed algorithm provides sets of eigenvalues and corresponding eigenvectors for the closed-loop subsystems that guarantee stability for arbitrary switching.
Several examples illustrate the characteristics of the problem considered and the application of the proposed design procedure.
\end{abstract}
 \vspace{0.5cm}\\


\section{Introduction}
\label{sec:intro}
Switched systems constitute a special kind of hybrid dynamical systems that consist of a family of subsystems with a switching signal that orchestrates between them.
Such systems appear for instance in modelling of mechanical systems, electric power systems, communication networks and intelligent control systems with logic-based controllers \cite{Lib03,SunGe05,SunGe11}. In particular the dynamics of the switched system may be discontinuous at the switching instances.
As such, the dynamics of the overall switched system are much more complex and may exhibit very different dynamical behaviour as can be found in the constituent subsystems.

For more than two decades the stability of switched linear systems has been subject of inspiring research \cite{LibMor99,DecBraPetLen00,ShoWirMasWulKin07,LinAnt09} and
monographs \cite{Lib03,SunGe05,SunGe11}.
However, constructive algebraic conditions for the stability of switched systems and thereby also efficient design-procedures that refrain from numerical tools remain scarce. 

In this paper we consider switched linear systems that switch autonomously in an arbitrary fashion, i.e.\ the switching is not restricted in neither time-domain nor state-space.
This work is dedicated to the study of feedback control design using eigenstructure assignment.
Such design method is well-established for linear time-invariant systems \cite{Moo76, Roppenecker.1990, Liu.1998, RasAhmSou09} and has already seen some successful transfer to switched linear systems in terms of pole-placement design \cite{WulWirSho05, WulWirSho09} and eigenstructure assignment approaches
\cite{BajFloKou13, Bajcinca.2015, HaiBra13, Hai16, WulHR21}.
In \cite{BajFloKou13, Bajcinca.2015}
subsystems that share a common right invariant eigenspace are considered and assign the left eigenspace of the closed-loop subsystems to achieve stability for arbitrary switching.
In \cite{HaiBra13, Hai16} simultaneous triangularization of the subsystems is achieved for subsystems with input matrices with transverse subspaces.
Both restrictions have been dropped in the approach presented in \cite{WulHR21}, where a simultaneous eigenstructure assignment for two third-order subsystems is proposed.
Thereby the closed-loop subsystems are simultaneously diagonalizable which ensures the existence of a common quadratic Lyapunov function for the Hurwitz subsystems \cite{MorMorKur98,ShoNar98} and hence, the closed-loop system is asymptotically stable for arbitrary switching \cite{Lib03}.

Assigning the eigenstructure of linear systems provides direct access to fundamental dynamical characteristics of the systems and allows to manipulate typical characteristics like response-time, damping or decoupling of certain dynamics.
In order to transfer such design methods to switched linear systems we aim for a constructive design-procedure that guarantees stability for arbitrary switching and provides sets of eigenvalues and corresponding eigenvectors from which the designer may choose.
Indeed exploiting the advantages of such eigenstructure assignment allows to decouple certain modes from the output such that a monotone output can be guaranteed for arbitrary switching \cite{HonSWR24}.
Such approach requires a parameterized description of assignable eigenvalues and eigenvectors which is not provided by results currently available.

In this contribution we generalize the results from \cite{WulHR21} for switched systems with two subsystems of arbitrary order.
While most concepts and results are easily generalized to more than two subsystems, we limit our analysis to two subsystems for ease of exposition.

The following section gives a formal problem statement in the context of the stabilization of switched linear systems.
We recall the seminal result of Moore \cite{Moo76} which gives necessary and sufficient conditions for the eigenstructure assignment problem for a single linear time-invariant system.
Our approach exploits this result in order to rectify the eigenstructure of two subsystems.
To this effect we introduce the notion of feedback rectifiable pairs, which states the existence of a feedback for every subsystem with the property that the closed-loop subsystems have the same eigenvectors; for details see Definition~\ref{def:fbrect} below.
It turns out that in some cases the eigenspaces of the subsystems may only be rectified if the assigned eigenvalues satisfy certain conditions, see Example~\ref{ex:lambda_mu}.
In Proposition \ref{prop:wieWHR} we formulate necessary and sufficient conditions for the feedback rectifiability of pairs for two subsystems.
Essentially the intersection of two subspaces determined by the subsystems has to span the whole space in order to allow for an arbitrary assignment of the eigenvalues.
This subspace can be described by a polynomial matrix as elaborated in Section \ref{sec:regular_case}.
The rectifiability can be characterized via properties of a polynomial matrix in Theorem~\ref{thm:main_real} and Theorem~\ref{thm:main_cplx}.
Further, we propose a constructive procedure to calculate the polynomial matrix in Corollary \ref{cor:schnitt_mit_echelon}.
With this characterization at hand, the eigenvalues and eigenvectors for the closed-loop subsystems can be chosen and the respective feedback matrices that stabilize the switched system for arbitrary switching are readily obtained, see Algorithm \ref{alg:F}.

Some notation used:  $\adj(C)^\top =(c_{ij})_{i,j=1}^n$ is the \emph{adjugate} of $C\in\C^{n\times n}$, $c_{ij}=(-1)^{i+j}\det C_{ij}$ and $C_{ij}$ is the matrix that is obtained from $C$ by deleting the $i$-th row and the \mbox{$j$-th} column.
The complex conjugate of a complex number~$\lambda$ is denoted by $\bar\lambda$.
For a matrix $A\in\R^{n\times n}$ we denote by $\sigma(A)\subset\C$ its set of eigenvalues. Furthermore, for a set $S\subseteq \mathbb{K}^n$ where $\mathbb{K}$ is either $\mathbb{R}$ or $\mathbb{C}$, we denote by ${\rm span}_{\mathbb{K}}$ the set of all finite linear combinations using coefficients in $\mathbb{K}$.

\section{Preliminaries and problem definition.}
\subsection{Problem definition}  
\label{sec:prob_def}

We consider the switched system that is composed by the constituent linear time-invariant (LTI) systems
\begin{align}\label{eq:sys_const_LTI}
	\dot x=A_q x+B_q u
\end{align}
with $A_q\in\mathbb R^{n\times n}$, $B_q\in\mathbb R^{n\times p}$ and $q\in\mathcal I=\{1,2\}$. 
The column rank of the input matrix $B_q$ is denoted by \mbox{$m_q:=\rk B_q\leq p$} and without loss of generality we assume that
\begin{align*} 
	p=m_1\geq m_2.
\end{align*} 

These subsystems constitute the switched system
\begin{align}  \label{eq:sys_switched_process}
	\dot x=A_{s(t)}x+B_{s(t)}u
\end{align}
where the switching signal $s:\mathbb R^+\rightarrow \mathcal I$ is a piecewise constant function. 
Note that the system and input matrices of \eqref{eq:sys_switched_process} switch simultaneously such that at any given time $t$ one of the subsystems~\eqref{eq:sys_const_LTI} is active.

We consider the asymptotic stability of the origin of the homogeneous
solution of \eqref{eq:sys_switched_process} for arbitrary switching signals of the above characteristics.
It is well known that asymptotic stability of the subsystems \eqref{eq:sys_const_LTI} does not imply asymptotic stability of the switched system for arbitrary switching.
Instead the existence of a common (but not necessarily quadratic) Lyapunov function for the subsystems is necessary and sufficient for the asymptotic stability for arbitrary switching \cite{LibMor99,DecBraPetLen00,ShoWirMasWulKin07}.

This paper aims at the control synthesis problem for the switched system \eqref{eq:sys_switched_process}.
Assuming the current value $s(t)$ of the switching signal is known at all times we consider the feedback law
\begin{align}   \label{eq:control_law}
	u = F_{s(t)}x
\end{align}%
with $F_{s(t)}\in\mathbb R^{p\times n}$.
The resulting closed-loop control system again constitutes a switched linear system and takes the form
\begin{align} \label{eq:sys_closedloop}
	\dot x = (A_{s(t)} + B_{s(t)}F_{s(t)})x\,.
\end{align}

The control problem considered in this contribution is to design feedback matrices $F_q\in\mathbb R^{p\times n}$, $q\in\mathcal I$ for a given switched system \eqref{eq:sys_switched_process} such that the closed-loop system \eqref{eq:sys_closedloop} is asymptotically stable for arbitrary switching and exhibits eigenvalues and eigenvectors to be chosen by the control designer.

\subsection{Preliminary results}
\label{sec:pre_results}

Our approach to design stable switched systems is inspired by the seminal results for the stability of switched systems for arbitrary switching \cite{MorMorKur98,ShoNar98}.
\begin{theorem}\label{thm:triangularisation}
	Given the switched system \eqref{eq:sys_switched_process} with $A_q$ Hurwitz, $q\in\mathcal I$.
	If there exists a $T\in\mathbb C^{n\times n}$, regular, such that $T^{-1}A_q T$ is in upper triangluar form for all $q\in\mathcal I$ then the origin of \eqref{eq:sys_switched_process} is asymptotically stable for arbitrary switching. Moreover there exists a common quadratic Lyapunov function.
\end{theorem}

Note that if two matrices commute, i.e.\ $A_1A_2=A_2A_1$ holds, then they are simultaneously triangulizable. More precisely, see \cite{DDG51}, simultaneous triangularisation holds true if and only if for every polynomial $p$ in two variables the matrix $p(A_1,A_2)
(A_1A_2-A_2A_1)$ is nilpotent. A~special case for which   Theorem~\ref{thm:triangularisation} can be applied is when $A_1$ and $A_2$ share the same set of $n$ linearly independent eigenvectors, and thus are simultaneously diagonalizable.

For our design of the switched controller \eqref{eq:control_law} we shall exploit the seminal result by Moore, \cite{Moo76},
that characterizes the eigenvalue and eigenvector assignment for a~single closed-loop LTI system $(A,B)$.
Consider $N(\lambda)$ and $M(\lambda)$ satisfying
\begin{align}\label{eq:K}
	\ker \begin{bmatrix}
		\lambda I-A&B    
	\end{bmatrix}=
	\im\begin{bmatrix}
		N(\lambda) \\
		M(\lambda)
	\end{bmatrix}.
\end{align}
Using this notion we have the following, \cite{Moo76}.
\begin{theorem}
	\label{thm:moore}
	Given $(A,B)\in\R^{n\times n}\times\R^{n\times p}$  and the self-conjugate set $\Lambda =\{\lambda_{1},\ldots,\lambda_{n}\}$  with distinct entries. Then there exists $F\in \mathbb{R}^{p\times n}$ such that $(A+BF)v_i=\lambda_i v_i$, $v_i\neq 0$, for all $i\in \{1,\ldots,n\}$ if and only if the following three conditions are satisfied:
	\begin{enumerate}
		\item[\rm (i)] $\{v_1,\ldots, v_n\}\subset \mathbb C^n$ are linearly independent vectors,
		\item[\rm (ii)] $v_l=\bar v_j$ whenever $\lambda_l=\bar\lambda_j$,
		\item[\rm (iii)] $v_i\in \im N(\lambda_i)$,
	\end{enumerate}
	where $\bar v$ and $\bar \lambda$ denote the complex conjugate of $v$ and $\lambda$, respectively.
\end{theorem}
\begin{remark}\label{GraciasAngelicaZapata}
	In Theorem \ref{thm:moore} the
	vectors $v_i$ are eigenvectors to the eigenvalues $\lambda_i$
	and, as the entries of $\Lambda$
	are distinct, the $v_i$ are automatically linear independent.
	Clearly,  Theorem \ref{thm:moore} still holds 
	under the assumption that the matrix $A+BF$
	is diagonalizable (that is, multiple eigenvalues are allowed but no generalized eigenvectors),
	see the remarks after Proposition 1 in \cite{Moo76}.
	Moreover, there exists also a version of Theorem \ref{thm:moore} without any assumptions
	on the Jordan structure of the matrix $A+BF$, that is,
	arbitrary Jordan chains are allowed, see \cite{klein77}.
\end{remark}
\begin{remark}
	Recall that $(A,B)$ is called controllable if $\rk [\lambda I-A,B]=n$ holds for all $\lambda\in\C$, which implies that the kernel in \eqref{eq:K} is of dimension $p$ for all $\lambda\in\C$. However, controllability is
	not mentioned in Theorem~\ref{thm:moore}. The equivalence of controllability and
	pole assignability implies that the three conditions in Theorem~\ref{thm:moore}
	cannot be satisfied if the uncontollable eigenvalues are not included in the selected set
	$\Lambda$ of closed-loop eigenvalues 
	(i.e., the numbers $\lambda_i$ in  Theorem~\ref{thm:moore}),
	see \cite{Moo76}.
\end{remark}

If $(A,B)$ is controllable then the kernel of $\begin{bmatrix}\lambda I-A & B\end{bmatrix}$ in \eqref{eq:K} can be described by matrices of dimension 
$N(\lambda)\in\mathbb R^{n\times p}$ and $M(\lambda)\in\mathbb R^{p\times p}$ with $\rk N(\lambda)=\rk B=m$ and $\rk M(\lambda)=p$.
For uncontrollable eigenvalues the kernel increases (see Section \ref{sec:uncontrollable_modes}).

Note that the set of assignable eigenvectors is given by the pre-image of the image of $B$ under the linear map $(\lambda I-A)$.
If not stated otherwise $N(\lambda)$ is to be understood as a parameterized representation of this pre-image (see also Example \ref{ex:lambda_mu}).
In any case the choice of $N(\lambda)$ and $M(\lambda)$ in \eqref{eq:K} is unique up to multiplications with invertible matrices from the right.
If $\lambda\in\C\setminus\sigma(A)$, then a possible choice is the following
\begin{align}
	\label{eq:N_via_adj}
	N(\lambda)=\adj(\lambda I -A)B,\quad M(\lambda)=-\det(\lambda I-A)I_p.
\end{align}
However, for $\lambda\in\sigma(A)$ the expression \eqref{eq:N_via_adj} is not valid, e.g.\ consider $A=B=0$, then $\adj(\lambda I -A)B=0$, but $\ker[-A,B]=\ker[0,0]=\R^2$.

\begin{remark}
		\label{rem:cont}
		In \cite{HaiOse2016} it is shown that the map of the eigenvalues to the kernel in \eqref{eq:K} is continuous in the gap metric for controllable $(A,B)$.
		This suggests that the map $N$ in \eqref{eq:N_via_adj} may be continuously extended for $\lambda\in\sigma(A)$ for such cases.
\end{remark}

In what follows, we shortly recall the main ideas of the proof of
Theorem~\ref{thm:moore}.
Assume that for some $F\in\R^{p\times n}$ and for 
$\{v_1, \ldots, v_n\}$ and  $\{\lambda_1,\ldots, \lambda_n\}$
as in Theorem~\ref{thm:moore}  the following holds
\begin{align*}
	(A+BF)v_i=\lambda_iv_i, \quad i\in\{1,\ldots,n\}.
\end{align*}
Then this implies
\begin{align}
	\label{eq:prep_intersect}
	\begin{bmatrix}
		v_i\\ -Fv_i
	\end{bmatrix}&\in\ker\begin{bmatrix}\lambda_iI-A & B\end{bmatrix}=\im\begin{bmatrix}
		N(\lambda_i)\\M(\lambda_i)
	\end{bmatrix}.
\end{align}
Therefore
$v_i\in \im N(\lambda_i)$ holds,  which shows the third
condition in Theorem~\ref{thm:moore}. The first two
conditions follow directly from elementary matrix theory.

Conversely,  if we have
linearly independent vectors $v_i$ with
$v_i=N(\lambda_i)w_i$ for some $w_i \in \mathbb R^p$, we can choose
the feedback matrix $F$ via
\begin{align}\label{eq:moore_calc_F}
	F
	\begin{bmatrix}
		v_1 &\cdots&v_n
	\end{bmatrix}
	=
	-\begin{bmatrix}
		M(\lambda_1)w_1& \cdots& M(\lambda_n)w_n
	\end{bmatrix},
\end{align}
which shows $(A+BF)v_i=\lambda_i v_i$. It remains to be shown
that $F$ can be chosen to be a real matrix, for this
we refer to  \cite{Moo76}.

\section{Feedback rectifiable pairs of matrices}
\label{sec:FB_rectifiable_pairs}
Our goal is to design feedback matrices $F_q$ such that the closed-loop subsystems \eqref{eq:sys_closedloop} are simultaneously similar to diagonal form and have eigenvalues with negative real parts.
In this section we investigate conditions for which two subsystems can be simultaneously diagonalized using state-feedback.

Proposition 1 in \cite{Hai16} provides a necessary and sufficient condition for simultaneous diagonalizability of an arbitrary number of subsystem for a chosen set of closed-loop eigenvalues.
This contribution aims towards providing a set of eigenvalues for which this property can be achieved.
This motivates the following notion.
\begin{defi}
	\label{def:fbrect}
	Let $A_q\in\R^{n\times n}$ and $B_q\in\R^{n\times p}$, $q=1,2$, and  let $D\subset\C\times\C$.
	Then the pairs $(A_1,B_1)$ and $(A_2,B_2)$ are called \emph{feedback rectifiable over $D$} if there exists $F_1,F_2\in\R^{p\times n}$, pairs of complex numbers $(\lambda_{1},\mu_1),\ldots,(\lambda_n,\mu_n)\in D$, and a linearly independent set $\{v_1,\ldots,v_n\}$ in $\C^n$ such that the following holds
	\begin{align}
		\label{def:ferec}
		(A_1+B_1F_1)v_i=\lambda_iv_i,\quad (A_2+B_2F_2)v_i=\mu_iv_i,
	\end{align}
	for all $i=1,\ldots,n$.
\end{defi}

Note that, contrary to the situation in
Theorem~\ref{thm:moore}, multiple eigenvalues
are allowed, see also
Remark~\ref{GraciasAngelicaZapata}.

With this notion together with Theorem \ref{thm:triangularisation} we readily obtain the following result \cite{WulHR21}.
\begin{theorem}
	\label{thm:stabilizable}
	The switched system \eqref{eq:sys_switched_process} is stabilizable for arbitrary switching via state-feedback \eqref{eq:control_law} if the pairs $(A_1,B_1)$ and $(A_2,B_2)$ are feedback rectifiable over a set $D$ that only contains elements with negative real parts.
\end{theorem}
In order to achieve rectifiability together with stabilization a sufficient number of parameters have to be available for the control design \cite{WulHR21}, which yields the following 
condition
\begin{align} \label{eq:nec_for_stabilization}
	m_1+m_2\geq n+1.    
\end{align}
Furthermore 
let
\begin{subequations} \label{eq:ker_AqBq}
	\begin{align}
		\label{eq:ker_A1B1}
		\ker \begin{bmatrix}
			\lambda I-A_1& B_1
		\end{bmatrix}&=\im\begin{bmatrix}N_1(\lambda)\\ M_1(\lambda)\end{bmatrix},\\
		\label{eq:ker_A2B2}
		\ker  \begin{bmatrix} \mu I-A_2 & B_2 \end{bmatrix}&=\im\begin{bmatrix}N_2(\mu)\\ M_2(\mu)\end{bmatrix},
	\end{align}
\end{subequations}
for some $N_1(\lambda),N_2(\mu)\in\C^{n\times p},M_1(\lambda),M_2(\mu)\in\C^{p\times p}$.
If the pairs $(A_1,B_1)$ and $(A_2,B_2)$ are feedback rectifiable over $D$, then \eqref{def:ferec} together with \eqref{eq:prep_intersect} implies
\begin{align}
	\label{eq:single_cap}
	v_i\in\im N_1(\lambda_i)\cap\im N_2(\mu_i),\quad \text{for all $i=1,\ldots,n$.}
\end{align}
Hence, we have  
\begin{align}\label{eq:span_of_all_intersections}
	\Span_{\C}\big((N_1\cap N_2)(D)\big)=\C^n,
\end{align}
where
\begin{align*}
	(N_1\cap N_2)(D):=\bigcup_{(\lambda,\mu)\in D}\im N_1(\lambda)\cap \im N_2(\mu).
\end{align*}

Moreover, we have the following characterization of feedback rectifiability which generalizes the result for $n=3$ from \cite[Theorem 4]{WulHR21}. 
\begin{prop}
	\label{prop:wieWHR}
	Consider the pairs $(A_q,B_q)\in\R^{n\times n}\times \R^{n\times p}$, $q=1,2$. Then the following holds:
	\begin{itemize}
		\item[\rm (a)] The pairs are feedback rectifiable over $D\subseteq\R\times\R$ if and only if $\Span_{\R}\big((N_1\cap N_2)(D)\big)=\R^n$ holds. 
		\item[\rm (b)]  If the pairs are feedback rectifiable over $D\subseteq\C\times\C$ then $\Span_{\C}\big((N_1\cap N_2)(D)\big)=\C^n$ holds. 
		\item[\rm (c)] The pairs are feedback rectifiable over $D\subseteq\C\times\C$ if there exists $\{(\lambda_1,\mu_1),\ldots,(\lambda_n,\mu_n)\}\subseteq D$ satisfying the following conditions:
		\begin{itemize}
			\item[\rm (i)] there exist $v_i\in \im N_1(\lambda_i)\cap\im N_2(\mu_i)$ for all $i=1,\ldots,n$ such that $\{v_1,\ldots,v_n\}$ is linearly independent;
			\item[\rm (ii)]  $v_i=\overline{v_j}$ holds whenever $\lambda_i=\overline{\lambda_j}$, and in this case we can choose without restriction $\mu_j=\overline{\mu_i}$.
		\end{itemize}
	\end{itemize}
\end{prop}
\begin{proof}
	(a) Assume that $\Span_{\R}\big((N_1\cap N_2)(D)\big)=\R^n$. Let $\{b_1,\ldots,b_n\}$ be a basis of $\R^n$. Then the assumption implies that there exist $\alpha_{i,j}\in\R$ with $i=1,\ldots,n$, $j=1,\ldots,k_i$ as well as elements 
	$v_{i,j}\in(N_1\cap N_2)(D)$ such that 
	\[
	b_i=\sum_{j=1}^{k_i}\alpha_{i,j}v_{i,j}
	\]
	holds for all $i=1,\ldots,n$. Hence $\R^n=\Span_{\R}\{v_{i,j}~|~ i=1,\ldots,n,~j=1,\ldots,k_i\}$. Therefore, this spanning set can be reduced to a basis of $\R^n$, say $\{v_{1},\ldots, v_{n}\}\subseteq (N_1\cap N_2)(D)$. Further, by definition of $(N_1\cap N_2)(D)$ there exist $(\lambda_i,\mu_i)\in D$ such that
	\begin{align}
		\label{vk_int}
		v_i\in \im N_1(\lambda_i)\cap \im N_2(\mu_i)\quad \text{for all $i=1,\ldots,n$.}
	\end{align}
	
	For the construction of feedback matrices $F_q\in\R^{p\times n}$, $q=1,2$, we can apply the construction proposed in \cite{Moo76}. 
	By the linear independence of its columns, $V:=\begin{bmatrix}
		v_1&\cdots&v_n
	\end{bmatrix}\in\R^{n\times n}$ is invertible. Further, invoking \eqref{vk_int} there exist $w_{qi}\in\R^{p}$ such that 
	\begin{align}
		\label{eq: parametervector}
		v_i=N_1(\lambda_i)w_{1i}=N_2(\mu_i)w_{2i},\quad \text{for all $i=1,\ldots,n$}.
	\end{align}
	
	Since $V^{-1}v_i$ equals the $i$-th canonical unit vector the feedback matrices
	\begin{subequations}
		\label{eq:construction_F_q}
		\begin{align}
			F_1&:=-\begin{bmatrix}   M_1(\lambda_1) w_{11} & \ldots &M_1(\lambda_n)w_{1n}
			\end{bmatrix}V^{-1}\\
			F_2&:=-\begin{bmatrix}   M_2(\mu_1) w_{21} & \ldots &M_2(\mu_n)w_{2n}
			\end{bmatrix}V^{-1}
		\end{align}
	\end{subequations}
	satisfy \eqref{def:ferec} as
	\begin{equation*}
		\begin{split}
			(A_1+B_1F_1&)v_i =
			(A_1-\lambda_iI+\lambda_iI+B_1F_1)v_i\\
			&= -(\lambda_iI -A_1)N_1(\lambda_i)w_{1i}+
			\lambda_i v_i+M_1(\lambda_i)w_{1i}\\
			&= \lambda_i v_i.
		\end{split}
	\end{equation*}
	A similar calculation shows 
	$(A_2+B_2F_2)v_i = \mu_i v_i$ which implies
	\eqref{def:ferec}.
	The remaining statement of (a) and the statement in (b) follow from the reasoning \eqref{eq:single_cap} preceding this proposition, which is true for both real and complex numbers.
	
	Assertion (c) follows 
	from Theorem \ref{thm:moore} and 
	Remark~\ref{GraciasAngelicaZapata}
	applied to each system $(A_1,B_1)$ and $(A_2,B_2)$.
\end{proof}

\begin{remark}
	Note that Theorem~\ref{thm:moore} cannot be directly applied to prove the converse of Proposition~\ref{prop:wieWHR}~(b). If we construct a complex-valued basis $\{v_1,\ldots,v_n\}\subseteq\C^n$ as in the proof of Proposition~\ref{prop:wieWHR}~(a) then this might lead to a set which does not contain pairs of complex conjugate vectors.
\end{remark}

Proposition~\ref{prop:wieWHR} is closely related to Proposition 1 in \cite{Hai16}, where a condition equivalent to \eqref{eq:span_of_all_intersections} is evaluated for a given set $\{(\lambda_1,\mu_1), \ldots, (\lambda_n,\mu_n)\}$. The set $D$ in \eqref{eq:span_of_all_intersections} may contain more pairs of assignable eigenvalues.

For our analysis we use a parameterized description of $N_1(\lambda)$ and $N_2(\mu)$ to obtain a parameterized description of the intersection $\im N_1(\lambda)\cap \im N_2(\mu)$ as illustrated and outlined by the following example.

\begin{ex}\label{ex:continuous}
	Consider the switched system \eqref{eq:sys_switched_process} with $(A_1, B_1)$ and $(A_2, B_2)$ given by 
	{\small 
		\begin{align*}
			A_1&=\begin{bmatrix} 0 & 0 & -1 & 0\\ 0 & 0 & 0 & 0\\ 0 & 0 & 0 & 0\\ 0 &1 & -2 & 0\end{bmatrix},   
			&&B_1=\begin{bmatrix}1 & 0 & 0\\ 1 & -1 & 1\\ 0 & 0 & -1\\ 0 & 0 & 0  \end{bmatrix}, \\
			A_2&=\begin{bmatrix} 0 & 0 & 1 & 0\\ 0 & 0 & 0 & 0\\ 0 & 0 & 0 & 0\\ 0 & 2 & 1 & 0\end{bmatrix},  
			&&B_2=\begin{bmatrix} 0 & 0& 0\\ 1 & 0& 0\\ 0 & 1& 1\\ 0 & 1& 1\end{bmatrix}.
		\end{align*}
	}%
	Note that $m_2=\rk B_2=2$ and the eigenvalues of the system matrices are given by $\sigma(A_1)=\sigma(A_2)=\{0\}$.
	Both subsystems are controllable.
	
	We obtain a representation of the kernels in \eqref{eq:ker_AqBq} using the adjugate representation in \eqref{eq:N_via_adj}:
	{\small 
		\begin{align*}
			N_1(\lambda) \!=\!\! 
			\begin{bmatrix} 
				\lambda ^3 & 0 & \lambda ^2\\[.5ex]
				\lambda ^3 & -\lambda ^3 & \lambda ^3\\[.5ex] 
				0 & 0 & -\lambda ^3\\[.5ex]
				\lambda ^2 & -\lambda ^2 & 3\lambda ^2 
			\end{bmatrix}\!\!,\:
			N_2(\mu)\!=\!\!\begin{bmatrix}
				0 & \mu ^2 & \mu ^2\\[.5ex]
				\mu ^3 & 0 & 0\\[.5ex]
				0 & \mu ^3 & \mu ^3\\[.5ex]
				2\,\mu ^2 & \mu ^3\!+\!\mu ^2 & \mu ^3\!+\!\mu ^2
			\end{bmatrix}\!.
		\end{align*}
	}%
	Note, that this representation is only valid for $\lambda,\mu\neq 0$.
	
	For each choice of $\lambda,\mu$ we obtain an intersection space.
	We are looking for the intersection space $\im N_1(\lambda)\cap \im N_2(\mu)$ parameterized in $\lambda,\mu$.
	This example shows that even though it is easy to find a representation for $N_1(\lambda)$ for $\lambda\notin \sigma(A_1)$ and $N_2(\mu)$  for $\mu\notin\sigma(A_2)$ it is not easy to calculate the intersections from \eqref{eq:single_cap} in order to get all assignable eigenvectors outside $\lambda\in\sigma(A_1)$ and $\mu\in\sigma(A_2)$, and to analyse if $n$ linearly independent eigenvectors can be chosen. This is necessary to calculate the feedback matrices $F_1$ and $F_2$ with \eqref{eq:construction_F_q}. Therefore we aim to give a constructive procedure with the concluding Algorithm \ref{alg:F}.
\end{ex}

From Proposition~\ref{prop:wieWHR}~(b) we have that the pairs $(A_q,B_q)$ are not feedback rectifiable over $D$ if the intersection $\im N_1(\lambda)\cap\im N_2(\mu)$ is empty for all $\lambda,\mu\in D$.
Therefore we obtain with \eqref{eq:N_via_adj} the following result.
\begin{cor}\label{cor:intersect_adj_necess}
	Let $A_1,A_2\in\R^{n\times n}$ and $B_1,B_2\in\R^{n\times p}$. Assume that there exists an $\alpha\in\C^n\setminus\{0\}$ such that $$\alpha^\top\adj(\lambda I -A_1)B_1=0$$ holds for all $\lambda\in\C$, or $$\alpha^\top\adj(\mu I -A_2)B_2=0$$ holds for all $\mu\in\C$. 
	
	Then 
	$(A_1,B_1)$ and $(A_2,B_2)$ are not feedback rectifiable over $(\C\setminus\sigma(A_1))\times (\C\setminus\sigma(A_2))$.
\end{cor}
The conditions in Corollary~\ref{cor:intersect_adj_necess} imply that at least one of the considered systems is not controllable, since \eqref{eq:prep_intersect} yields that $\alpha$ is also orthogonal to all potentially assigned eigenvectors.

In the following low-dimensional example we consider feedback rectifiability
with respect to different underlying sets $D\subset
\mathbb C \times \mathbb C$.

\begin{ex}
	\label{ex:lambda_mu}
	Consider the controllable pairs $(A_1,B_1)$ and $(A_2,B_2)$ with
	\begin{align*}
		A_1=\begin{bmatrix}-1&0\\0&1\end{bmatrix}=-A_2,\;\;\text{and}\;\;
		B_1=\begin{bmatrix}
			1\\1
		\end{bmatrix}\!,\, B_2=\begin{bmatrix}
			1\\-1
		\end{bmatrix}.
	\end{align*}
	First, we consider the case $\lambda_1=\mu_1=1$ and $\lambda_2=\mu_2=-1$, i.e.\ we aim to assign closed-loop eigenvalues within the spectrum of $A_q$.
	Then the kernel \eqref{eq:K} is given by
	\begin{align*}
		\ker\begin{bmatrix}\lambda_1 I -A_1 &B_1\end{bmatrix}&=\ker\begin{bmatrix}
			2&0&1\\0&0&1
		\end{bmatrix},\\
		\ker\begin{bmatrix}\lambda_2 I -A_1 & B_1\end{bmatrix}&=\ker\begin{bmatrix}
			0&0&1\\0&-2&1
		\end{bmatrix},\\
		\ker\begin{bmatrix}\mu_1 I -A_2 & B_2\end{bmatrix}&=\ker\begin{bmatrix}
			0&0&1\\0&2&-1
		\end{bmatrix},\\ 
		\ker\begin{bmatrix}\mu_2 I -A_2 & B_2\end{bmatrix}&=\ker\begin{bmatrix}
			-2&0&1\\0&0&-1
		\end{bmatrix}\,.
	\end{align*}
	We can write the upper parts of the kernels as
	\begin{align}\label{eq:ex_lambda=mu=pm1}
		N_1(\lambda_1)=\begin{bmatrix} 0\\1 \end{bmatrix}=N_2(\mu_2),\;\;\;
		N_1(\lambda_2)=\begin{bmatrix} 1\\0 \end{bmatrix}=
		N_2(\mu_1).
	\end{align}
	Hence, we have $\im N_1(\lambda_1)\cap\im N_2(\mu_1)=\{0\}$ and $\im N_1(\lambda_2)\cap\im N_2(\mu_2)=\{0\}$ and thus, the 
	pairs are not feedback rectifiable over $D=\{(1,1),(-1,-1) \}$.
	However, swapping the pairing of the eigenvalues we find that they are feedback rectifiable over $D=\{(1,-1),(-1,1) \}$.
	
	For $\lambda,\mu\in\C\setminus\{-1,1\}$ we obtain the upper parts of the kernels according to \eqref{eq:N_via_adj} 
	\begin{subequations}
		\label{eq:ex_lambda=mu}
		\begin{align}
			N_1(\lambda)&=(\lambda I-A_1)^{-1}B_1=\begin{bmatrix}(\lambda +1)^{-1}\\ (\lambda -1)^{-1}\end{bmatrix},\\
			N_2(\mu)&=(\mu I-A_2)^{-1}B_2=\begin{bmatrix}(\mu -1)^{-1}\\ -(\mu +1)^{-1}\end{bmatrix}.
		\end{align}
	\end{subequations}
	
	For a nontrivial intersection $\im N_1(\lambda)\cap \im N_2(\mu)$ there must exist a non-zero $\beta\in\C$ satisfying \[
	\begin{bmatrix}(\lambda +1)^{-1}\\ (\lambda -1)^{-1}\end{bmatrix}=\beta \begin{bmatrix}(\mu -1)^{-1}\\ (-\mu -1)^{-1}\end{bmatrix},
	\]
	or equivalently
	\[
	\frac{\lambda+1}{\mu-1}=\frac{1-\lambda}{\mu+1}\quad \Leftrightarrow\quad  \lambda\mu=-1.
	\]
	Hence we have co-linearity of $N_1(\lambda)$ and $N_2(\mu)$ for \mbox{$\mu=-\lambda^{-1}$.}
	It can be verified using \eqref{eq:ex_lambda=mu=pm1} and \eqref{eq:ex_lambda=mu} that 
	the intersection for all $\lambda\in\C\setminus\{-1,1,0\}$ is given by
	\[
	\im N_1(\lambda)\cap\im N_2(-\lambda^{-1})
	=\left\{\begin{bmatrix}
		\lambda-1\\ \lambda+1
	\end{bmatrix}\right\},
	\]
	which yields linearly independent vectors. 
	Hence, whenever 
	the set $D$ contains two distinct (self-conjugate) pairs $(\lambda_1,\mu_1)$, $(\lambda_2,\mu_2)$
	with 
	\mbox{$\lambda_1\mu_1=\lambda_2\mu_2=-1$} then the pairs 
	$(A_1, B_1)$
	and $(A_2, B_2)$ 
	are feedback rectifiable over $D$,
	e.g.\
	$D=
	\{ (e^{i\phi}, -e^{-i\phi}), (e^{-i\phi}, -e^{i\phi})\}$ for $\phi\in [0,2\pi)$.
\end{ex}
Example~\ref{ex:lambda_mu} shows that rectified eigenstructures may only be achieved with some additional condition for the eigenvalues of the closed-loop systems, e.g.\  $\lambda_i\mu_i=-1$. 
But $\lambda_i\mu_i=-1$ implies $\mu_i=-\lambda_i^{-1}$. Hence
\begin{align*}
	\mbox{\rm Re}\, \mu_i =
	\mbox{\rm Re}\,(-\lambda_i^{-1}) =
	- |\lambda_i|^{-2} \mbox{\rm Re}\, \lambda_i
\end{align*}
and therefore it is not possible to choose both closed-loop eigenvalues, $\lambda_i$ and $\mu_i$, to have negative real-parts.
Hence the stabilization of the switched system with rectified eigenstructure may not be possible.

\section{Characterization of feedback rectifiability}
\label{sec:regular_case}

In this section we apply Proposition~\ref{prop:wieWHR} to develop conditions for which two pairs of matrices $(A_1,B_1)$ and $(A_2,B_2)$ are feedback rectifiable. To this end we derive characterizations such that $(N_1\cap N_2)(D)$ contains a basis of $\R^n$ or $\C^n$. We compute the intersection of $\im N_1(\lambda)$ and $\im N_2(\mu)$ along the lines from \cite{PizOH99,Yan97}. The following lemma is a slight generalization (see also Remark~\ref{rem:lem41}).

\begin{prop}
	\label{prop:intersection}
	Let $\mathbb{F}$ be an arbitrary field and consider the subspaces $\mathcal{K}=\im K$ and $\mathcal{L}=\im L$ for some $K\in\mathbb{F}^{n\times r_1}$ and $L\in\mathbb{F}^{n\times r_2}$. Then there exists an invertible $R\in\mathbb{F}^{n\times n}$ satisfying 
	\begin{align}
		\label{eq:R_L_decomp}
		RK=\begin{bmatrix} K_1\\ 0
		\end{bmatrix},\quad \rk K_1=n_1,\quad RL=\begin{bmatrix}
			L_1\\ L_2
		\end{bmatrix}
	\end{align}
	for some $K_1\in\mathbb{F}^{n_1\times r_1}$, $L_1\in\mathbb{F}^{n_1\times r_2}$ and $L_2\in\mathbb{F}^{(n-n_1)\times r_2}$ and the intersection of the subspaces is given by 
	\begin{align}
		\label{eq:lem_intersection}
		\mathcal{K}\cap\mathcal{L}=\im K\cap\im L=R^{-1}\begin{bmatrix}
			L_1\\ 0
		\end{bmatrix}(\ker L_2).
	\end{align}
	
	Consider $r_3:=\rk L_2$. If $r_3=0$ then the following holds
	\begin{align}
		\label{eq:r3_is_zero}
		\mathcal{K}\cap\mathcal{L}=\im R^{-1}\begin{bmatrix}
			L_1\\0 
		\end{bmatrix}.\quad  
	\end{align}
	Moreover, if $r_3=r_2>0$ then  $\mathcal{K}\cap\mathcal{L}=\{0\}$ holds.
	
	If $0<r_3<r_2$ holds then there exist invertible $\hat R\in\F^{n\times n}$ and $\hat V\in \F^{r_2\times r_2}$ so that we can further decompose \eqref{eq:R_L_decomp} as follows 
	\begin{align*}
		\hat R K&=\begin{bmatrix}
			K_1\\0\\0
		\end{bmatrix}\!\!,\;\; 
		\hat RL\hat V=\begin{bmatrix}
			L_{11} & L_{12}\\ L_{21}& L_{22}\\ 0&0
		\end{bmatrix}\!\!,\;\;
		L_1\hat V=\begin{bmatrix}
			L_{11} & L_{12}
		\end{bmatrix}\!\!,
	\end{align*}
	such that $L_{21}\in\F^{r_3\times r_3}$ is invertible,  $L_{12}\in\F^{n_1\times(r_2-r_3)}$ and $L_{22}\in\F^{r_3\times(r_2-r_3)}$. Furthermore, the following holds 
	\begin{align}
		\label{eq:intersection_quasi_rref}
		\hspace{-.5em}
		\mathcal{K}\cap\mathcal{L}=\im K\cap\im L=\im \hat R^{-1}\!\begin{bmatrix}
			L_{12}\!-\!L_{11}L_{21}^{-1}L_{22}\\ 0\\ 0
		\end{bmatrix}\!.
	\end{align}
\end{prop}

\begin{proof}
	Applying Gaussian elimination there exist an invertible matrix $R$
	and matrices $K_1, L_1$, and $L_2$ such that \eqref{eq:R_L_decomp}
	holds. As $K_1$ has full row-rank it is surjective. Therefore the following holds
	\[
	\im(RK)\cap\im(RL)=\im \begin{bmatrix}
		K_1\\0
	\end{bmatrix}\cap\im \begin{bmatrix}
		L_1\\ L_2
	\end{bmatrix}=\begin{bmatrix}
		L_1\ker L_2\\0
	\end{bmatrix}.
	\]
	This together with the invertibility of $R$ yields
	\begin{align}\notag
		\im K\cap\im L&=R^{-1}\left(\im(RK)\cap \im(RL)\right)\\
		&= R^{-1}\begin{bmatrix}
			L_1\ker L_2\\0
			\label{eq:R_raus}
		\end{bmatrix},
	\end{align}
	which shows \eqref{eq:lem_intersection}. If $r_3=\rk L_2=0$ holds, then $\ker L_2=\R^{r_2}$ and therefore \eqref{eq:lem_intersection} implies \eqref{eq:r3_is_zero}. 
	
	If $r_3>0$ then, applying Gaussian elimination to the rows of $L_2$, we obtain for some invertible $W\in\F^{(n-n_1)\times(n-n_1)}$ and some invertible $\hat V\in\F^{r_2\times r_2}$
	\begin{align*}
		\begin{bmatrix}
			I_{n_1}&0\\0&W
		\end{bmatrix}RK&=\begin{bmatrix}
			K_1\\0\\0
		\end{bmatrix},\\
		\begin{bmatrix}
			I_{n_1}&0\\0&W
		\end{bmatrix}RL\hat V&=\begin{bmatrix}
			L_{11}& L_{12}\\ L_{21} & L_{22}\\0&0
		\end{bmatrix},
	\end{align*}
	with $\rk L_{21}=\rk L_2=r_3$.
	Here $L_{21}\in\F^{r_3\times r_3}$ is chosen to be square and therefore invertible. 
	Defining \mbox{$\hat R:=\begin{bmatrix}
			I_{n_1}&0\\0&W
		\end{bmatrix}R$}, we have
	\begin{align*}
		\im(\hat R K)\cap\im (\hat RL)&=
		\im(\hat R K)\cap\im (\hat RL \hat V)\\
		&=\im\begin{bmatrix}
			K_1\\0\\0
		\end{bmatrix}\cap\im\begin{bmatrix}
			L_{11}L_{21}^{-1} & L_{12}\\ 
			I_{r_3} & L_{22}\\
			0&0
		\end{bmatrix}\\
		&=\im\begin{bmatrix}
			L_{12}-L_{11}L_{21}^{-1}L_{22}\\ 0\\0
		\end{bmatrix}.
	\end{align*}
	This together with a repetition of the steps in \eqref{eq:R_raus} proves \eqref{eq:intersection_quasi_rref}.
\end{proof}
\begin{remark}\label{rem:lem41}
	In \cite{PizOH99,Yan97} the invertible transformation $R$ in Proposition~\ref{prop:intersection} is chosen in such a way that $[RK,RL]$ is in \emph{reduced row echelon form}. This is a particular staircase form resulting from Gaussian elimination where the leading entries in each row are equal to one. The columns corresponding to such leading entries are unit vectors, i.e.\ all of their other entries are equal to zero. In particular, we have $L_{11}=0$ and $L_{21}=I_{r_3}$ in \eqref{eq:intersection_quasi_rref}. Hence, if $\hat R$ is chosen in such a way that $[\hat RK,\hat RL]$ is in reduced row echelon form then the intersection \eqref{eq:intersection_quasi_rref} is given by
	\begin{align}
		\label{eq:rref_intersection}
		\im K\cap\im L=\im \hat R^{-1}\begin{bmatrix}
			L_{12}\\ 0
		\end{bmatrix}.
	\end{align}
\end{remark}

In the following we apply Proposition~\ref{prop:intersection} to compute the intersection of the parameter dependent subspaces $\im N_1(\lambda)$ and $\im N_2(\mu)$ given by~\eqref{eq:N_via_adj}. To this end, we consider the \emph{quotient field} $\R(\lambda,\mu)$ of the ring $\R[\lambda,\mu]$ of polynomials $p(\lambda,\mu)=\sum_{i=0}^I\sum_{j=0}^Jp_{ij}\lambda^i\mu^j$ with real-valued coefficients $p_{ij}\in\R$,
see e.g.\ \cite[Section IV.\S 7]{Lan13}, i.e.\ the quotient field $\R(\lambda,\mu)$ simply contains all rational functions in two variables with real-valued coefficients.

Applying Proposition~\ref{prop:intersection} using $\F=\R(\lambda,\mu)$ to \begin{align*}K&=N_1(\lambda)=(\lambda I-A_1)^{-1}B_1,\\ L&=N_2(\mu)=(\mu I-A_2)^{-1}B_2,\end{align*} there exist matrix functions
$R(\lambda,\mu), \hat R(\lambda,\mu)\in\R(\lambda,\mu)^{n\times n}$ and 
$L_1(\lambda,\mu),L_2(\lambda,\mu),K_1(\lambda,\mu)$ 
of appropriate sizes such that \eqref{eq:R_L_decomp} 
and \eqref{eq:intersection_quasi_rref} hold.

This leads to a parameter-dependent representation 
\begin{align}
	\label{eq:thm_int_new}
	\hspace{-0.2cm}\im N_1(\lambda)\cap \im N_2(\mu)=\im Q(\lambda,\mu)
\end{align}
in terms of a rational matrix function $Q(\lambda,\mu)\in\R(\lambda,\mu)^{n\times r}$, which is determined depending on the ranks $r_2$ and $r_3$ as follows
\begin{align}
	\label{eq:def_Q_1}
	Q(\lambda,\mu):=\begin{cases}
		R(\lambda,\mu)^{-1}\!
		\begin{bmatrix}
			L_1(\lambda,\mu)\\0
		\end{bmatrix}\!,& \text{if $r_3=0$},\\[4ex]
		0,& \text{if $r_3=r_2>0$,}\\
		\hat{R}(\lambda,\mu)^{-1}\!\begin{bmatrix}
			\hat L_1(\lambda,\mu)\\ 0\\ 0
		\end{bmatrix}\!,& \text{if $r_2>r_3>0$,}
	\end{cases}
\end{align}
with $\hat L_1(\lambda,\mu)=L_{12}(\lambda,\mu)-L_{11}(\lambda,\mu)L_{21}(\lambda,\mu)^{-1}L_{22}(\lambda,\mu)$.

It is important to note that~\eqref{eq:thm_int_new} has to be understood as an equality of subspaces of rational functions where $\lambda$ and $\mu$ are not determined. Therefore, there might exist $(\lambda_0,\mu_0)\in\C\times\C$ for which $\im Q(\lambda_0,\mu_0)\neq\im N_1(\lambda_0)\cap\im N_2(\mu_0)$. Nevertheless, Proposition~\ref{prop:intersection} can still be applied, but using $\mathbb{F}=\mathbb{R}$ or $\mathbb{F}=\mathbb{C}$. In the following, we determine a set
\[
\Omega_0:=\big((\C\setminus\sigma(A_1))\times(\C\setminus\sigma(A_2))\big)\setminus \mathcal{P}(Q),
\]
where $\mathcal{P}(Q)$ is the \emph{set of poles} of the  rational matrix-valued function $Q(\lambda,\mu)$, i.e.\ the set of all $(\lambda_0,\mu_0)\in\C\times\C$ for which at least one of the entries of $Q(\lambda,\mu)$ has a pole at $(\lambda,\mu)=(\lambda_0,\mu_0)$.

In the following proposition, we summarize the above findings and it is shown that the set of exceptional points where the intersection cannot be computed using \eqref{eq:thm_int_new} is small.
\begin{prop}
	\label{cor:rref}
	Let $A_q\in\R^{n\times n}$, $B_q\in\R^{n\times p}$, $q=1,2$ and
	consider the two matrix-valued functions 
	$N_q$ which are given by \eqref{eq:N_via_adj} and the rational matrix in two variables $Q(\lambda,\mu)\in\R(\lambda,\mu)^{n\times r}$ given by \eqref{eq:def_Q_1}. Then there exists a dense set $\Omega\subseteq(\C\setminus\sigma(A_1))\times(\C\setminus\sigma(A_2))$ such that for all $(\lambda,\mu)\in \Omega$ the following holds 
	\begin{align}
		\label{eq:thm_int}
		\hspace{-0.2cm}\im N_1(\lambda)\cap \im N_2(\mu)=\im Q(\lambda,\mu).
	\end{align}
\end{prop}
\begin{proof}
	The set $\Omega$ is given by
	\begin{align*}
		\!\!\Omega:=\begin{cases}
			\Omega_0, & \begin{matrix}\text{if $r_3=0$ or~~~~} \\  \text{~~~$r_2>r_3>0$,}\end{matrix}\\ 
			\begin{matrix*}[l]\big\{(z,w)\in\Omega_0:\rk R(z,w)=n, \\ ~~~~~~~~~~~~~~~~~~~~~\rk L_2(z,w)=r_2 \big\},\end{matrix*} & \text{if $r_3=r_2>0$,} 
		\end{cases}
	\end{align*}
	where we have to impose additional rank assumptions for $r_3=r_2>0$ to ensure the invertibility of $R(\lambda,\mu)$ and $L_2(\lambda,\mu)$ at the considered points.
	It remains to prove that $\Omega$ is dense. If  $r_3=0$ or $r_2>r_3>0$ hold, then $\Omega=\Omega_0$ and the following holds
	\begin{align*}
		\mathbb C \setminus \Omega \subset \mathcal{P}(Q) \cup
		\bigcup_{\lambda \in \sigma(A_1)} \{\lambda\} \times \mathbb C \cup
		\bigcup_{\mu \in \sigma(A_2)}\mathbb C \times  \{\mu\}.
	\end{align*}
	Obviously, the sets $\{\lambda\} \times \mathbb C$, $\mathbb C \times  \{\mu\}$ and $\mathcal{P}(Q)$ 
	have an empty interior. Therefore, the interior of $\C\setminus\Omega$ is empty which implies that $\Omega=\Omega_0$ is dense. If $r_2=r_3>0$ holds, then $\Omega$ is given by the intersection of a dense open set $\Omega_0$ with the dense set 
	\[
	\big\{(z,w)\in\C\times\C : \rk R(z,w)=n,\, \rk L_2(z,w)=r_2 \big\}
	\]
	and therefore $\Omega$ is dense in this case.
\end{proof}

For the dense set $\Omega$ constructed in Proposition~\ref{cor:rref} we consider also the following set 
\begin{align}
	\label{def_D_real}
	\Omega_{\R}&:=\Omega\cap(\R\times\R).
\end{align}   
It follows immediately from the construction of the set $\Omega$, that $\Omega_\R$ is also dense in $\R\times\R$.

For further analysis we shall rewrite the rational function $Q(\lambda,\mu)$ with entries $r_{i,j}(\lambda,\mu)=\tfrac{p_{i,j}(\lambda,\mu)}{q_{i,j}( \lambda,\mu)}$
satisfying~\eqref{eq:thm_int} as a polynomial matrix $P(\lambda,\mu)\in\R[\lambda,\mu]^{n\times r}$ on the set $\Omega$, which is obtained by multiplying each column of $Q(\lambda,\mu)$ with the least common multiple of all denominator polynomials $q_{i,j}(\lambda,\mu)$ of all of the entries of the selected column.  
More precisely, define $q_j(\lambda,\mu):=\prod_{i=1}^{n}q_{i,j}(\lambda,\mu)$, $j=1,\ldots,r$. Then $P\in\R[\lambda,\mu]^{n\times r}$ is explicitly given by
\begin{align}
	\label{def_P_from_Q}
	P(\lambda,\mu)=Q(\lambda,\mu)
	\begin{bmatrix}
		q_1(\lambda,\mu)&&0\\[-0.5ex]
		&\!\!\ddots\!\!&\\[-0.5ex]
		0& &q_{r}(\lambda,\mu)
	\end{bmatrix}.
\end{align} 

To characterize feedback rectifiability, we study the mapping $P\in\R[\lambda,\mu]^{n\times r}$ given by \eqref{def_P_from_Q} and consider the representation
\begin{align}
	\label{def:P}
	P(\lambda,\mu)=\sum_{k=0}^N\sum_{l=0}^NC_{kl}\lambda ^k\mu^l
\end{align}
for some $C_{kl}\in\R^{n\times r}$ and $N\geq0$. Furthermore we consider the union of the images of $P\in\R[\lambda,\mu]^{n\times r}$ evaluated on some set $G\subset\C\times\C$
\[
\Sp_G(P):=\Span\left(\bigcup_{(\lambda,\mu)\in G} \im P(\lambda,\mu)\right)\subseteq\C^n.
\]
Let $G\subseteq(\C\setminus\sigma(A_1))\times(\C\setminus\sigma(A_2))$ and choose $P$ as in \eqref{def_P_from_Q} then according to 
\eqref{eq:thm_int} the following holds
\begin{align}
	\label{spd_schnitt}
	\Sp_G(P)=\Span\big((N_1\cap N_2)(G)\big).
\end{align}

In the lemma below, we characterize when $\Sp_G(P)<n$ holds in terms of the matrix polynomial $P$. 

\begin{lemma}
	\label{lem:notfull}
	Let $P\in \R[\lambda,\mu]^{n\times r}$ and $G\subseteq\C\times \C$. Then the following holds.
	\begin{itemize}
		\item[\rm (a)]  If there exists $\alpha\in\C^n\setminus\{0\}$ such that $\alpha^\top P(\lambda,\mu)=0$ holds for all $(\lambda,\mu)\in G$ then $\dim\Sp_G(P)<n$.
		\item[\rm (b)] Let $G$ be a dense subset of $\R\times\R$.  If $\dim\Sp_G(P)<n$ then there exists $\alpha\in\R^n\setminus\{0\}$ such that $\alpha^\top P(\lambda,\mu)=0$ holds for all $(\lambda,\mu)\in\R\times\R$.
	\end{itemize}
\end{lemma}
\begin{proof}
	If $\alpha\in\C^n\setminus\{0\}$ fulfills the assumption in (a) then $\alpha$ is orthogonal to $\Sp_G(P)$. Hence $\dim\Sp_G(P)<n$. We continue with the proof of (b). If $\dim\Sp_G(P)<n$ then there exists $\alpha\in\C^n$ such that $\alpha^\top P(\lambda,\mu)=0$ holds for all $(\lambda,\mu)\in G$. Since $G$ is dense in $\C\times\C$, for all  $(\lambda,\mu)\in G$ there exists a sequence $((\lambda_l,\mu_l))_{l\in\N}$ in $G$ converging to $(\lambda,\mu)$. Then the continuity of $P$ implies for all $(\lambda,\mu)\in G$
	\[
	\alpha^\top P(\lambda,\mu)=\lim\limits_{l\rightarrow\infty}\alpha^\top P(\lambda_l,\mu_l)=0
	\]
	which proves (b). 
\end{proof}

Below we present the first main result. 
\begin{theorem}
	\label{thm:main_real}
	Let $A_q\in\R^{n\times n}$, $B_q\in\R^{n\times m_q}$, $q=1,2$ and let $P\in\R[\lambda,\mu]^{n\times r}$ and $\Omega_{\R}$ satisfy \eqref{def_D_real} and assume that $P$ is as in \eqref{def:P} for some matrices $C_{kl}\in\R^{n\times r}$. 
	
	The pairs $(A_q,B_q)$ are feedback rectifiable over $\Omega_{\R}$ if and only if the following condition holds
	\begin{align}
		\label{eq:ker_P}
		\ker \begin{bmatrix}
			C_{00} & \ldots& C_{0N}& C_{10}&\ldots& C_{NN}
		\end{bmatrix}^\top=\{0\}.
	\end{align}
\end{theorem}
\begin{proof}
	Let \eqref{eq:ker_P} hold. If there exists  $\alpha\in\R^n\setminus\{0\}$ satisfying $\alpha^\top P(\lambda,\mu)=0$ for all $(\lambda,\mu)\in\R\times\R$. Then with \eqref{def:P} we get
	\begin{align}
		\label{eq:lincomzero}
		\alpha^TP(\lambda,\mu)=\sum_{k=0}^N\sum_{l=0}^N\alpha^TC_{kl}\lambda ^k\mu^l=0\in\R^{1\times r}.
	\end{align}
	Further, in the vector space $\R[\lambda,\mu]$ over $\R$, the set of polynomials $\{\lambda ^k\mu^l\}_{k,l=0}^N$ is linearly independent. Hence \eqref{eq:lincomzero} implies 
	\[
	\alpha^TC_{kl}=0\quad \text{for all $k,l=1,\ldots,N$,}
	\]
	but this contradicts \eqref{eq:ker_P}. 
	Hence, there exists no $\alpha\in\R^n\setminus\{0\}$ satisfying $\alpha^TP(\lambda,\mu)=0$ for all $(\lambda,\mu)\in\R \times \R$. 
	Therefore, by \eqref{spd_schnitt} and Lemma~\ref{lem:notfull}~(b) and Proposition~\ref{cor:rref}, we conclude $\Span(N_1\cap N_2)(\Omega_{\R})=\Sp_{\Omega_\R}(P)=\R^n$. 
	Thus, Proposition~\ref{prop:wieWHR}~(a) implies that the pairs $(A_q,B_q),\; q=1,2$ are feedback rectifiable over~$\Omega_{\R}$.
	
	If the pair is feedback rectifiable over $\Omega_{\R}$ then Proposition~\ref{prop:wieWHR}~(a) and \eqref{spd_schnitt} lead to $\Span(N_1\cap N_2)(\Omega_{\R})=\R^n$, but then $\R^n\supseteq\Sp_{\R\times\R}(P)\supseteq\Sp_{\Omega_{\R}}(P)=\R^n$ implies that $\Sp_{\R\times\R}(P)=\R^n$ and therefore there exists no $\alpha\in\R^n\setminus\{0\}$ satisfying $\alpha^TP(\lambda,\mu)=0$ for all $(\lambda,\mu)\in\C\times\C$. As shown above, this implies \eqref{eq:ker_P}.
\end{proof}

If we combine Theorem~\ref{thm:main_real} with the fact that feedback rectifiability over $\Omega_{\R}$ also implies feedback rectifiability over the larger set $\Omega$, and, by applying Proposition~\ref{prop:wieWHR} (b) we obtain the following necessary and sufficient condition.
\begin{cor}
	Let $A_q\in\R^{n\times n}$, $B_q\in\R^{n\times m_q}$, $q=1,2$ and let $\Omega$ by given by Proposition~\ref{cor:rref} and assume that $P\in\R[\lambda,\mu]^{n\times r}$ is as in \eqref{def:P}. 
	Then the pairs $(A_q,B_q)$ are feedback rectifiable over $\Omega$ if and only if \eqref{eq:ker_P} holds.
\end{cor}

Note in particular, if \eqref{eq:ker_P} does not hold for the set $\Omega_\R$, then the pairs are also not feedback rectifiable over any extension $\Theta\supset\Omega_\R$.

\section{Construction of the feedback rectification}\label{sec:construction}

In this section we consider pairs of matrices for which 
\begin{align}
	\label{eq:ranktildeN}
	\rk\begin{bmatrix}
		N_1(\lambda) &N_2(\mu)
	\end{bmatrix}=n
\end{align}
holds for some points $(\lambda,\mu)\in\C\times\C$.
We obtain an explicit formula for the intersection $\im N_1(\lambda) \cap \im N_2(\mu)$ which allows us to specify a set $D$ for rectification and to formulate a constructive algorithm for the design of the feedback matrices.

In particular consider pairs of matrices satisfying the following assumption.

\begin{assumption}
	\label{ass:ordering}
	Let $A_q\in\R^{n\times n}$, $B_q\in\R^{n\times m_q}$, $q=1,2$,  let \eqref{eq:nec_for_stabilization} be valid and let $\Theta\subset\C\times\C$.
	We assume that condition \eqref{eq:ranktildeN} holds for all $(\lambda,\mu)\in \Theta$.
\end{assumption}
Note that \eqref{eq:ranktildeN} holds for $n=3$ and $m_1=m_2=2$ and controllable subsystems,  \cite{WulHR21}. 
Hitherto we have not seen a practically relevant example with $n>3$, for which Assumption \ref{ass:ordering} does not hold.

Define 
\begin{align}
	\label{def:Nhat}
	\hat N(\lambda,\mu):=\begin{bmatrix} N_1(\lambda) & N_2(\mu)\!\begin{bmatrix} I_{n-p}\\0
	\end{bmatrix}\end{bmatrix} \!\!\in\R^{n\times n}.
\end{align}
Without loss of generality we can reorder the columns of $B_1$ and $B_2$ such that  $\hat N(\lambda,\mu)$ is invertible for all $(\lambda,\mu)\in\Theta$, if condition \eqref{eq:ranktildeN} holds.

The matrix $\hat N(\lambda,\mu)$ in \eqref{def:Nhat} is used in the following corollary to compute the intersection of $\im N_1(\lambda)$ and $\im N_2(\mu)$ based on Remark~\ref{rem:lem41}. 

\begin{cor} \label{cor:schnitt_mit_echelon}
	Let $A_q\in\R^{n\times n}$ and $B_q\in\R^{n\times p}$, $q=1,2$, and $(\lambda,\mu)\in\Theta$ satisfy Assumption~\ref{ass:ordering}. 
	Then the following holds.
	\begin{itemize}
		\item[\rm (a)] We have:
		$\dim\big(\im N_1(\lambda)\cap\im N_2(\mu)\big)\geq 1$. 
		\item[\rm (b)]  For $L_{12}\in\R^{m_1\times(p-r_2)}$ and $L_{22}\in\R^{(n-m_1)\times(p-r_2)}$ given by 
		\begin{align} \label{eq:def_L12}
			\begin{bmatrix}
				L_{12}\\ L_{22}
			\end{bmatrix}
			:=\hat N(\lambda,\mu)^{-1}\begin{bmatrix}
				N_1(\lambda) & N_2(\mu)
			\end{bmatrix}
			\begin{bmatrix}
				0\\I_{p-r_2}
			\end{bmatrix}
		\end{align}
		we have $L_{12}\neq 0$ and the following holds
		\begin{align}
			\label{eq:int_with_Nhat}
			\im N_1(\lambda)\cap\im N_2(\mu)=\im \hat N(\lambda,\mu)\!\begin{bmatrix}
				L_{12}\\0
			\end{bmatrix}\!.
		\end{align}
	\end{itemize}
\end{cor}
\begin{proof}
	For $(\lambda,\mu)\in\Theta$ for which Assumption \ref{ass:ordering} holds, we obtain the following estimate
	\[
	\dim\ker\begin{bmatrix}
		N_1(\lambda) & N_2(\mu)
	\end{bmatrix} =2p-n
	\geq p+m_2-n\geq 1.
	\]
	As a consequence, there exist nonzero $x,y\in\R^p$ fulfilling $N_1(\lambda)x=-N_2(\mu)y$. Since $\rk B_1=p=m_1$ holds, we have that $N_1(\lambda)x\neq 0$ holds which proves (a).
	
	To prove (b), we use the definition of the inverse and Assumption~\ref{ass:ordering} which implies 
	\begin{align}\label{eq:def_part_echolon}
		\hat N(\lambda,\mu)^{-1}\begin{bmatrix}
			N_1(\lambda) & N_2(\mu)
		\end{bmatrix}
		=\begin{bmatrix}
			I_{p}&0& L_{12}\\0&\!\!I_{n-p}\!\!&L_{22}
		\end{bmatrix}\!\!.
	\end{align}
	Hence, the intersection is given by
	\begin{multline*}
		\hat N(\lambda,\mu)^{-1}\big(\im N_1(\lambda)\cap\im N_2(\mu)\big)\\
		=\im \begin{bmatrix}
			I_p\\ 0
		\end{bmatrix}\cap\im \begin{bmatrix}
			0& L_{12}\\ I_{n-p}& L_{22}
		\end{bmatrix} 
	\end{multline*}
	which yields \eqref{eq:int_with_Nhat}.
\end{proof}

Note that in Corollary~\ref{cor:schnitt_mit_echelon} the point $(\lambda,\mu)\in\C\times\C$ is fixed whereas in Proposition~\ref{cor:rref} the intersection of  $N_1(\lambda)$ and $N_2(\mu)$ is determined as a function of $\lambda$ and $\mu$ on a dense set $\Omega$, which might exclude certain points such as $\lambda\in\sigma(A_1)$ and $\mu\in \sigma(A_2)$.
In Assumption~\ref{ass:ordering} and Corollary~\ref{cor:schnitt_mit_echelon} we consider a possibly larger set $\Theta$, which may also contain the eigenvalues of $A_1$ and $A_2$ as discussed in Remark \ref{rem:cont}. Note that for $(\lambda,\mu)\in\Omega\cap\Theta$ the images of the  matrix $Q(\lambda,\mu)$ used in Proposition~\ref{prop:intersection} and the matrix on the right hand side in \eqref{eq:int_with_Nhat} coincide.

Next, we use Proposition~\ref{prop:wieWHR} (c) to obtain a sufficient condition for feedback rectifiability over $\Theta\subset \C\times\C$.
\begin{theorem}
	\label{thm:main_cplx}
	Let $A_q\in\R^{n\times n}$, $B_q\in\R^{n\times p}$, $q=1,2$. 
	Let there exist $\{(\lambda_1,\mu_1),\ldots,(\lambda_n,\mu_n)\}\subseteq\Theta$ 
	satisfying Assumption~\ref{ass:ordering} such that for all non-real $(\lambda_l,\mu_l)$ there exists some $j$ satisfying $(\lambda_l,\mu_l)=(\overline{\lambda_j},\overline{\mu_j})$ and consider
	\begin{align}\label{eq:def_E}
		E_i:=\hat N(\lambda_i,\mu_i)\begin{bmatrix}
			L_{12}\\ 0
		\end{bmatrix},\; i=1,\ldots,n,
	\end{align}
	with $\hat N(\lambda_i,\mu_i)$ as given by \eqref{def:Nhat} and $L_{12}$ in \eqref{eq:def_L12}. Then, the pairs $(A_q,B_q)$, are feedback rectifiable over $\Theta$ if 
	\begin{align}
		\label{eq:ker_Ei}    
		\ker\begin{bmatrix}
			E_{1} & E_{2}& \ldots& E_{n}\end{bmatrix}^\top=\{0\}.
	\end{align}
\end{theorem}
\begin{proof}
	From \eqref{eq:int_with_Nhat} we find for all $i=1,\ldots,n$
	\[
	\im N_1(\lambda_i)\cap\im N_2(\mu_i)=\im \hat N(\lambda_i,\mu_i)\!\begin{bmatrix}
		L_{12}\\0\end{bmatrix}=\im E_i.
	\]
	This, together with \eqref{eq:ker_Ei} implies 
	$\im\!\begin{bmatrix}
		E_1 & \ldots & E_n
	\end{bmatrix}=\C^n$. 
	Hence, Proposition~\ref{prop:wieWHR}~(c) implies that the pairs are feedback rectifiable over $\Theta$. 
\end{proof}

With the previous results and the construction of the intersection \eqref{eq:int_with_Nhat} at hand we obtain the following procedure for the design of rectifying feedback.
\begin{alg}
	\label{alg:F}
	If the pairs $(A_1,B_1)$ and $(A_2,B_2)$ are feedback rectifiable and satisfy Assumption~\ref{ass:ordering} then the feedback matrices $F_q$ satisfying equation \eqref{def:ferec} can be constructed as follows.
	\begin{itemize}
		\setlength{\itemindent}{2em}
		\item[\bfseries{Step 1}] Calculate the subspaces $\im\begin{bmatrix}N_q(\lambda)^\top& M_q(\lambda)^\top\end{bmatrix}^\top$, $q=1,2$ in \eqref{eq:ker_AqBq}.
		\item[\bfseries{Step 2}] Determine $\hat{N}(\lambda,\mu)$ and $\Theta\subset\C\times\C$ in \eqref{def:Nhat}.
		\item[\bfseries{Step 3}] Calculate the intersection \eqref{eq:int_with_Nhat} following Corollary \ref{cor:schnitt_mit_echelon} and check whether \eqref{eq:ker_P} holds.
		\item[\bfseries{Step 4}] Choose the eigenvalues $(\lambda_i,\mu_i)\in \Theta$ for both closed-loop subsystems.
		\item[\bfseries{Step 5}] Choose $c_i\in \C^{p-r_2}$, $i=1,\ldots n$, such that \mbox{$v_i=N_1(\lambda_i)L_{12}(\lambda_i,\mu_i)c_i$} are linear independent eigenvectors of the closed-loop system satisfying \eqref{eq:int_with_Nhat}.
		\item[\bfseries{Step 6}]    The parameter vectors $w_{1i}\in\R^p$ in equation \eqref{eq: parametervector} are obtained by $w_{1i}=L_{12}(\lambda_i,\mu_i)c_i$.  
		\item[\bfseries{Step 7}] The parameter vectors $w_{2i}$ for the second system can be found by solving $v_i=N_2(\mu_i)w_{2i}$.
		\item[\bfseries{Step 8}] Calculate the feedback matrices $F_1$ and $F_2$ as in~\eqref{eq:construction_F_q}. 
	\end{itemize}
\end{alg}

We continue Example \ref{ex:continuous} to illustrate the application of the results obtained in this section.
We consider the stabilization of a fourth-order switched system with three inputs and two modes.
In order to calculate the intersection we resort to the reduced row echelon form utilized in Corollary~\ref{cor:schnitt_mit_echelon}.
The polynomial matrix representation \eqref{def:P} is used to verify the rectifiability of the two pairs via Theorem \ref{thm:main_cplx}.
The procedure given by Algorithm \ref{alg:F} is then applied to  choose the eigenvectors and obtain the feedback matrices of the control law \eqref{eq:control_law}.

\begin{ex}
	Consider the switched system \eqref{eq:sys_switched_process} with system matrices and $N_1(\lambda)$, $N_2(\mu)$ in Example~\ref{ex:continuous}.
	In order to calculate the intersections of the spaces spanned by $N_1(\lambda)$ and $N_2(\mu)$ we employ Corollary \ref{cor:schnitt_mit_echelon}.
	Note that $m_1+m_2\geq n+1$
	and combining $N_1(\lambda)$ with the first column of $N_2(\mu)$ we obtain an invertible $\hat N(\lambda,\mu)$ for $\Theta = \big(\C\setminus\sigma(A_1)\times\C\setminus\sigma(A_2)\big)\setminus\big\{(\lambda,\mu)\,|\, 2\lambda=\mu\big\}$. 
	This yields 
	{\small 
		\begin{align*}
			\hat N(\lambda,\mu)^{-1}\!=\!
			\begin{bmatrix}
				\frac{1}{\lambda ^3} & 0 & \frac{1}{\lambda ^4} & 0\\[1ex]
				\frac{1}{\lambda ^3} & -\frac{2}{\lambda ^2 (2\,\lambda -\mu )} & \frac{2\,\lambda -\mu +3\,\lambda \,\mu -2\,\lambda ^2}{\lambda ^4\,(2\,\lambda -\mu )} & \frac{\mu }{\lambda ^2\,(2\,\lambda -\mu )}\\[1ex]
				0 & 0 & -\frac{1}{\lambda ^3} & 0\\[1ex]
				0 & -\frac{1}{\mu ^2\,(2\,\lambda -\mu )} & \frac{2}{\mu ^2\,(2\,\lambda -\mu )} & \frac{\lambda }{\mu ^2\,(2\,\lambda -\mu )}
			\end{bmatrix}\!.
		\end{align*}
	}%

	The reduced row echelon form of $\begin{bmatrix}
		N_1(\lambda)& N_2(\mu)
	\end{bmatrix}$ is obtained  by multiplication with $\hat N(\lambda,\mu)^{-1}$, i.e.\ 
	{\small 
		\begin{multline*}
			\begin{bmatrix}\adj(\lambda I-A_1)B_1& \adj(\lambda I-A_2)B_2\end{bmatrix}=\\
			=    \hat N(\lambda,\mu)^{-1} 
			\begin{bmatrix}
				N_1(\lambda)& N_2(\mu)
			\end{bmatrix}\\
			=\begin{bmatrix} 
				1 & 0 & 0 & 0 & \frac{\mu ^2(\lambda +\mu )}{\lambda ^4} & \\[1ex]
				0 & 1 & 0 & 0 & \frac{\mu ^2(\lambda ^2\mu ^2-\lambda ^2\mu +2\,\lambda ^2+3\lambda \,\mu ^2+\lambda \mu -\mu ^2)}{\lambda ^4(2\lambda -\mu )} & \quad\star\quad \\[1ex]
				0 & 0 & 1 & 0 & -\frac{\mu ^3}{\lambda ^3} & \\[1ex]
				0 & 0 & 0 & 1 & \frac{\lambda +2\mu +\lambda \mu }{2\lambda -\mu } & 
			\end{bmatrix}\!,
		\end{multline*}
	}%
	where $\star$ denotes the repetition of the second last column.
	
	Identifying the upper $m_1=3$ entries of the last two columns in the above matrix with $L_{12}$ in \eqref{eq:int_with_Nhat} we obtain the intersection by 
	{\small 
		\begin{multline}
			\label{eq:SchnittExample}
			\im N_1(\lambda)\cap\im N_2(\mu)\!=\! \im \hat N(\lambda,\mu)\begin{bmatrix}
				L_{12}\\0
			\end{bmatrix}\\
			\!=\!\Span\left\{\left[\begin{array}{cc}  \mu ^2 & \mu ^2\\[1ex]
				-\frac{\mu ^3    (\lambda +2    \mu +\lambda     \mu )}{2    \lambda -\mu } & -\frac{\mu ^3    (\lambda +2    \mu +\lambda     \mu )}{2    \lambda -\mu }\\[1ex]
				\mu ^3 & \mu ^3\\[1ex]
				-\frac{\mu ^3    (\mu +5)}{2    \lambda -\mu } & -\frac{\mu ^3    (\mu +5)}{2    \lambda -\mu }  \end{array}\right]\right\}.
		\end{multline}
	}%
	
	We can now choose pairs of eigenvalues to be assigned $\{(\lambda_1,\mu_1),\ldots, (\lambda_4,\mu_4)\}$, self-conjugate with negative real parts, satisfying $2\lambda_i\neq \mu_i$.
	Substitute these pairwise into the matrix in \eqref{eq:SchnittExample} to obtain the matrices $E_i$ in \eqref{eq:def_E} of Theorem \ref{thm:main_cplx}.
	If the matrix $\begin{bmatrix}
		E_1&E_2&E_3&E_4 
	\end{bmatrix}$ 
	has full rank, we can simultaneously diagonalize the two subsystems assigning the closed-loop eigenvalues chosen.
	Continue with Step 5 of Algorithm \ref{alg:F} to choose common eigenvectors for the two subsystems and obtain the corresponding feedback matrix.
	
	Note that such heuristic strategy may not succeed for some systems, see Example \ref{ex:defective}.
	However, for this example we can use Theorem \ref{thm:main_real} to prove that the pairs $(A_1, B_1)$,$(A_2, B_2)$ are feedback rectifiable over $\Omega_\R=\big((\R\setminus\sigma(A_1))\times(\R\setminus\sigma(A_2))\big)\setminus\{(\lambda,\mu)\,|\,2\lambda=\mu\}$.
	
	The rational matrix-valued function in \eqref{eq:SchnittExample} is a representation of $Q(\lambda,\mu)$ in~\eqref{eq:thm_int}.
	Accordingly the set of poles $\mathcal{P}(Q)$ consists of the roots of the polynomial $(2\lambda-\mu)$ which are excluded from the set $\Theta$.
	Note that the computation of the intersection valid for $(\lambda,\mu)\in D_p:=\big\{(\R\setminus\sigma(A_1))\times(\R\setminus\sigma(A_2))|\,2\lambda=\mu\big\}$ may yield feedback rectifiability for points excluded in $\Omega_\R$, see Example~\ref{ex:defective}.

	The polynomial matrix in \eqref{def:P} is given by 
	{\small 
		\begin{align} \label{eq:ex_P}
			P(\lambda, \mu)&=
			\begin{bmatrix}  \mu ^2\,(2\,\lambda -\mu )\\ 
				-\mu ^3\,(\lambda +2\,\mu +\lambda \,\mu )\\ 
				\mu ^3\,(2\,\lambda -\mu )\\ 
				-\mu ^3\,(\mu +5)\end{bmatrix}\\[1ex] \notag
			&=C_{03}\mu^3 + C_{04}\mu^4 + C_{12}\lambda\mu^2 +  C_{13} \lambda\mu^3+C_{14} \lambda \mu^4
		\end{align}
	}%
	with 
	{\small 
		\begin{align*}
			C_{03}\!=\!\begin{bmatrix}-1\\0\\0\\-5\end{bmatrix}\!,
			C_{04}\!=\!\begin{bmatrix}
				0\\-2\\-1\\-1
			\end{bmatrix}\!,
			C_{12}\!=\!\begin{bmatrix}
				2\\0\\0\\0
			\end{bmatrix}\!,
			C_{13}\!=\!\begin{bmatrix}0\\ -1\\2\\0\end{bmatrix}\!,
			C_{14}\!=\!\begin{bmatrix}
				0\\-1\\0\\0
			\end{bmatrix}\!.
		\end{align*}
	}%
	The matrix $\begin{bmatrix}
		C_{03} &C_{04} &C_{12} &C_{13} &C_{14} 
	\end{bmatrix}$ has full rank.
	Thus by Theorem \ref{thm:main_real}, the two pairs are feedback rectifiable over $\Omega_\R$. 
	Since $\Omega_\R\subset\Theta$ we have feedback rectifiability over $\Theta$. 
	It follows immediately, that the switched system~\eqref{eq:sys_switched_process} is stabilizable via rectification, c.f. Theorem~\ref{thm:stabilizable}.

	Next we design the control law \eqref{eq:control_law} that stabilizes the switched system.
	Therefore we constructed the feedback matrices $F_q$ from equation \eqref{def:ferec} using Algorithm \ref{alg:F}.
	
	We choose the eigenvalues of the closed-loop system~\eqref{eq:sys_closedloop} as $\lambda_1=-3$, $\lambda_2=-1$, $\lambda_3=-2$, $\lambda_4=-4$ for mode $q=1$, and $\mu_1=-1$,  $\mu_2=-3$, $\mu_3=-2$ and  $\mu_4=-4$ for mode $q=2$.
	The assignable eigenvectors $v_i$ can be chosen from $\im P(\lambda_i,\mu_i)$.
	Substituting $(\lambda_i,\mu_i)$ into \eqref{eq:ex_P} yields
	{\small 
		\begin{align*}
			v_1=\begin{bmatrix}
				5\\2\\-5\\-4
			\end{bmatrix}, \quad v_2=\begin{bmatrix}
				1\\-12\\-3\\6
			\end{bmatrix}, \quad v_3=\begin{bmatrix}
				1\\2\\-2\\-3
			\end{bmatrix}, \quad v_4=\begin{bmatrix}
				1\\-4\\-4\\-1
			\end{bmatrix}.
		\end{align*}
	}%
	
	Due to the dimensions of the system in this example the eigenvector for each pair $(\lambda_i,\mu_i)$ is uniquely defined. 
	For higher-order systems with larger number of inputs the intersection space for each pair of eigenvalues may be increased, such that $P(\lambda_i,\mu_i)$ in~\eqref{eq:ex_P} has several columns.
	In such case the eigenvector associated with a pair of eigenvalues can be chosen from a set parameterized by the vector $c_i\in \dom P(\lambda_i,\mu_i)$.
	
	Using this vector we determine the parameter vectors~$w_{1i}$ in \eqref{eq: parametervector} for subsystem $q=1$
	\begin{align*}
		w_{1i}=L_{12}(\lambda_i,\mu_i)c_i
	\end{align*}
	resulting in
	{\small 
		\begin{align*}
			w_{11}=\begin{bmatrix}
				-\frac{4}{81}\\[1ex] -\frac{29}{405}\\[1ex] -\frac{1}{27} 
			\end{bmatrix}\!,\;  w_{12}=\begin{bmatrix}
				-36\\[1ex] -171\\[1ex] -27 
			\end{bmatrix}\!,\;  w_{13}=\begin{bmatrix}
				-1\\[1ex] -1\\[1ex] -1 
			\end{bmatrix}\!,\;  w_{14}=\begin{bmatrix}
				-\frac{1}{2}\\[1ex] -\frac{5}{2}\\[1ex] -1 
			\end{bmatrix}\!.
		\end{align*}
	}%
	By solving $v_i=N_2(\lambda_2)w_{2i}$ for the second system we obtain the parameter vectors 
	{\small 
		\begin{align*}
			w_{21}=\begin{bmatrix}
				-\frac{2}{5}\\[1ex] 1\\[1ex] 0
			\end{bmatrix}\!,\;   w_{22}=\begin{bmatrix}
				4\\[1ex] 1\\[1ex] 0
			\end{bmatrix}\!,\;   w_{23}=\begin{bmatrix}
				-1\\[1ex] 1\\[1ex] 0
			\end{bmatrix}\!,\;   w_{24}=\begin{bmatrix}
				1\\[1ex] 1\\[1ex] 0 
			\end{bmatrix}\!.
		\end{align*}
	}%
	
	Note that the last two columns of $N_2(\mu)$ are linearly dependent and therefore the last column does not provide any degree of freedom. 
	As a result, the last entry of each $w_{2i}$ is calculated to zero.
	
	Finally, the feedback gains \eqref{eq:moore_calc_F} for each subsystem are calculated to 
	{\small 
		\begin{align*}
			F_1\!=\!\begin{bmatrix}
				-\frac{29}{2} & 14 & -\frac{41}{2} & \frac{39}{2}\\[1ex]
				-\frac{337}{4} & 105 & -\frac{609}{4} & \frac{579}{4}\\[1ex]
				-\frac{63}{2} & 39 & -\frac{113}{2} & \frac{109}{2}
			\end{bmatrix}\!\!,\;
			F_2\!=\! \begin{bmatrix}\frac{15}{4} & -8 & \frac{27}{4} & -\frac{29}{4}\\[1ex]
				\frac{21}{2} & -17 & \frac{43}{2} & -\frac{47}{2}\\[1ex]
				0 & 0 & 0 & 0\end{bmatrix}\!,
		\end{align*}
	}%
	providing the closed-loop system matrices $A_q+B_qF_q$, \mbox{$q=1,2$}, with the eigenvalues and eigenvectors chosen above.
	
	Note again that the last column of the input matrix~$B_2$ does not provide any further degrees of freedom. 
	This results in the zero row of the feedback matrix $F_2$.
	In fact, $F_2$ includes a standard input transformation which leads to an input matrix with full rank.
\end{ex}

\begin{ex}
	\label{ex:defective}
	The following example demonstrates that even though the system is not feedback rectifiable by constructing the intersection by \eqref{eq:int_with_Nhat} of Corollary \ref{cor:schnitt_mit_echelon} it might be feedback rectifiable over another set $D_p$.
	
	Consider the systems $(A_1, B_1)$ and $(A_2 ,B_2)$ given by 
	{\small 
		\begin{align*}
			A_1&=\begin{bmatrix}0 & 0 & 1 & 0\\ 0 & 2 & 0 & 0\\ 0 & 0 & 0 & 1\\ 0 & -1 & 2 & 1\end{bmatrix}\!, &&   B_1=\begin{bmatrix} 1 & 0 & 0\\ 0 & 1 & 0\\ 0 & 0 & 0\\ 0 & 0 & 1 \end{bmatrix}\!,\\
			A_2&=\begin{bmatrix}  1 & 0 & 0 & 0\\ 0 & 0 & 0 & -2\\ 0 & 0 & 0 & 1\\ 1 & 2 & 2 & 1 \end{bmatrix}\!, &&  B_2=B_1 .
		\end{align*}
	}%
	Choosing the approach of Corollary \ref{cor:schnitt_mit_echelon}, the intersection can be computed with the transformation matrix
	\begin{align}\label{eq:ex_defect_Nhat}
		\hat N(\lambda,\mu)^{-1}=\begin{bmatrix} X&Y\end{bmatrix}
	\end{align}
	with
	{\small
		\begin{align*}
			X&=\begin{bmatrix}
				\frac{1}{(\lambda +1)\,{(\lambda -2)}^2}    & 0 \\[1ex]
				0 & \frac{1}{\lambda (\lambda+1)(\lambda-2)} \\[1ex]
				0 & \frac{1}{\lambda (\lambda +1){(\lambda -2)}^2} \\[1ex]
				0 & 0   
			\end{bmatrix}\\[1.5ex]
			Y&=\begin{bmatrix}
				\frac{-\lambda ^2\,\mu ^2+\lambda ^2\,\mu -2\,\lambda ^2+\mu }{\lambda \,(\lambda -\mu )\,(\lambda +1)\,{(\lambda -2)}^2} & \frac{\lambda \,\mu ^2-\lambda \,\mu +2\,\lambda -1}{\lambda \,(\lambda -\mu )\,(\lambda +1)\,{(\lambda -2)}^2}\\[1ex]
				\frac{2}{(\lambda -\mu )\,(\lambda+1)(\lambda-2)} & \frac{-2}{\lambda \,(\lambda -\mu )\,(\lambda+1)(\lambda-2)}\\[1ex]
				\frac{2\,\lambda +2\,\mu +\lambda \,\mu -\lambda ^2\,\mu }{\lambda \,(\lambda -\mu )(\lambda+1)(\lambda-2)^2} & \frac{\lambda ^2-\lambda -4}{\lambda \,(\lambda -\mu )(\lambda+1)(\lambda-2)^2}\\[1ex]
				\frac{\lambda }{\mu \,(\lambda -\mu )} & \frac{-1}{\mu \,(\lambda -\mu )}    
			\end{bmatrix}.
		\end{align*}%
	}%
	Note that $\hat N(\lambda,\mu)$ is regular for $(\lambda,\mu)\in \Omega=(\C\setminus\sigma(A_1)\times\C\setminus\sigma(A_2))\setminus\mathcal{P}(Q)$ with $$\mathcal{P}(Q)=\big\{(\lambda,\mu)~|~\lambda \mu(\lambda+1)(\lambda-2)(\lambda-\mu)=0\big\}.$$
	Following the Algorithm \ref{alg:F} we obtain the intersection
	$\im N_1(\lambda)\cap\im N_2(\mu)=\im Q(\lambda,\mu)$ with
	\begin{align*}
		Q(\lambda,\mu)&= (\mu -1 )  (\mu ^2-\mu +2 ) \left[\begin{array}{cc}      2 & \mu  \\ - 1 & 0\\  0 & 0\\  0 & 0 
		\end{array}\right].
	\end{align*}
	In this case $Q(\lambda,\mu)$ coincides with $P(\lambda,\mu)$ in \eqref{def_P_from_Q}.
	Obviously Condition \eqref{eq:ker_P} of Theorem \ref{thm:main_real} is not satisfied and thus it is not possible to select $n$ pairs $(\lambda_i,\mu_i)$ to generate $n$ linearly independent eigenvectors from the obtained intersection.
	
	Now we consider the kernel with the following restriction for the eigenvalues to be chosen: $\mu_i=\lambda_i$.
	We consider the new set 
	$$D_p=\left\{(\lambda,\mu)\in\{\C\setminus\sigma(A_1)\times\C\setminus\sigma(A_2)\}\,\big|\,\lambda=\mu\right\}.$$
	Note that $\hat N(\lambda,\mu)$ in \eqref{eq:ex_defect_Nhat} cannot be used to compute the intersection for $(\lambda,\mu)\in D_p$.
	But we can compute the intersection $\im N_1(\lambda)\cap\im N_2(\mu)$ via the  reduced  row echelon form of $\begin{bmatrix}
		N_1(\lambda) & N_2(\mu) \end{bmatrix}$ to 
	{ \small 
		\begin{multline*}
			\im P(\lambda,\lambda)=\\
			\Span\left\{\!
				\begin{bmatrix}
					-\lambda  ({\lambda}^2\!-\!\lambda\!+\!2 ) & 0 & 0\\ 
					2 \lambda &
					-(\lambda\!-\!1)(\lambda\!+\!1)(\lambda\!-\!2)
					& 2 \lambda  (\lambda\!-\!1 ) \\ 
					-\lambda & -2(\lambda\!-\!1) & -\lambda  (\lambda\!-\!1 ) \\
					-{\lambda}^2 & -2 \lambda  (\lambda\!-\!1 ) & -{\lambda}^2  (\lambda\!-\!1 )
				\end{bmatrix}\!\right\}\!.
			\end{multline*}
		}%
		%
		%
		%
		%
		Thus we obtain
		{ \small 
			\begin{align*}
				P(\lambda, \lambda)=
				C_0 + C_1\lambda + C_2\lambda^2 + C_3\lambda^3    
			\end{align*}
			with
			\begin{align*}
				C_0&=   \begin{bmatrix}
					0&0&0\\
					0&-2&0\\
					0& 2&0\\
					0&0&0
				\end{bmatrix},
				&&C_1=  \begin{bmatrix}
					-2 &0&0\\
					2&1&-2\\-1&-2&1\\0&2&0
				\end{bmatrix},\\
				C_2&= \begin{bmatrix}1&0&0\\0&2&2\\0&0&-1\\-1&-2&1\end{bmatrix},
				&& C_3= \begin{bmatrix}-1&0&0\\0&-1&0\\0&0&0\\0&0&-1\end{bmatrix}.
			\end{align*}
		}%
		As the $C_i$ above satisfy \eqref{eq:ker_P}, the given pairs are feedback rectifiable over $D_p$.
		In fact, \eqref{def:ferec} can be satisfied choosing any four distinct pairs $(\lambda_i,\mu_i)\in D_p$, and thus the switched system is indeed stabilizable via feedback rectification.
		
		Since $\dim \im P(\lambda,\lambda)=3$, the intersection provides a three-dimensional space for each $\lambda$ to choose the eigenvector from.
		For $\lambda_1=\mu_1=-1$ the corresponding eigenvectors can be chosen as a linear combination of the columns in
		\begin{align*}
			P(-1,-1)=
			\begin{bmatrix}
				4 & 0 & 0 \\
				-2 & 0 & 4\\
				1 & 4 & -2\\
				1 & -4 & 2        
			\end{bmatrix}.
		\end{align*}
		Choosing $c_1 = [1\;\;1\;\;1]^\top$, for instance, yields the eigenvector $v_1=[4\;\; 2\;\;3\;\; -1]^\top$.
		
	\end{ex}

\section{A note on uncontrollable modes}\label{sec:uncontrollable_modes}
The procedure proposed in Section \ref{sec:regular_case} can be applied to controllable as well as uncontrollable modes of the subsystems $(A_q,B_q)$.
Of course the uncontrollable eigenvalues have to be eigenvalues of the closed-loop system $A_q+B_qF_q$ and therefore cannot be re-assigned.
However the corresponding eigenvector is not affected and can be chosen within the limits of Theorem \ref{thm:main_cplx}.
In fact, the dimension of the kernel from which the eigenvector for the single subsystem can be chosen increases, as can be seen as follows.

Consider $(A_1,B_1)$ with the unique uncontrollable eigenvalue $\lambda_u$ such that
\begin{align*}
	\rk \begin{bmatrix}
		\lambda_u I -A_1 & B_1
	\end{bmatrix} = n -\nu(\lambda_u) < n,
\end{align*}
where $\nu(\lambda_u)$ denotes the geometric multiplicity of $\lambda_u$, i.e. the dimension of the eigenspace associated with the so-called uncontrollable mode $\lambda_u$.
Then $$\dim \ker \begin{bmatrix}
	\lambda_u I -A_1 & B_1
\end{bmatrix} = p + \nu(\lambda_u).$$

Note that the dimension of the kernel at $\lambda=\lambda_u$ is larger than the kernel for all other $\lambda\in\C$.
Therefore we describe the upper part of the kernel associated with the uncontrollable mode using $N_1^u(\lambda_u)\in\R^{n\times p+\nu(\lambda_u)}$ whereas for all $\lambda\neq\lambda_u$ it is described by $N_1(\lambda)\in\R^{n\times p}$.
This increase of dimension, however, does not cause any further restrictions on the rectifiablility of the pairs $(A_1,B_1)$, $(A_2,B_2)$ as the intersection of the spaces in Proposition \ref{prop:wieWHR} also increases as the following example shows.

\begin{ex} \label{ex:uncontrollable}
	Consider the system with uncontrollable but stable eigenvalue $\lambda_u=\mu_u=-1$ in each mode with
	{\small 
		\begin{align*}
			A_1=\begin{bmatrix}
				-1 & 0 & 0\\ 0 & 0 & 0\\ 0 & 0 & 1     
			\end{bmatrix}\!\!,\; 
			A_2=\begin{bmatrix} 
				-1 & 0 & 0\\ 1 & 0 & 0\\ 0 & 0 & 1  
			\end{bmatrix}\!\!,\;
			B_1=B_2=\begin{bmatrix}0 & 0\\ 0 & 1\\ 1 & 0     \end{bmatrix}\!.
		\end{align*}
	}%
	For the controllable eigenvalues we have
	{\small 
		\begin{align*}
			N_1(\lambda)\!=\!\begin{bmatrix} 0 & 0\\ 0 & 1\!-\!\lambda ^2\\ -\lambda   (\lambda \!+\!1 ) & 0  \end{bmatrix}\!,
			\;
			N_2(\mu)\!=\!\begin{bmatrix} 0 & 0\\ 0 & 1\!-\!\mu ^2\\ -\mu   (\mu \!+\!1 ) & 0  \end{bmatrix}
		\end{align*}
	}
	and 
	{\small 
		\begin{align*}
			M_1(\lambda)=\begin{bmatrix} \lambda ^3\!-\!\lambda  & 0\\ 0 & \lambda ^3\!-\!\lambda  \end{bmatrix}\!,
			\;
			M_2(\mu)=\begin{bmatrix} \mu ^3\!-\!\mu  & 0\\ 0 & \mu ^3\!-\mu  \end{bmatrix}\!.
		\end{align*}
	}%
	Observe, that both subsystems have a zero row in the kernel matrices $N_q$ corresponding to the uncontrollable mode and thus their intersection will not span the first component in $\R^3$.
	
	For the uncontrollable modes $\lambda_u=\mu_u=-1$ we have 
	{\small 
		\begin{align*}
			N^u_1(\lambda_u)=\begin{bmatrix} 1 & 0 & 0\\[0.5ex] 0 & 0 & \frac{\sqrt{2}}{2}\\[0.5ex] 0 & \frac{\sqrt{5}}{5} & 0 \end{bmatrix}\!,
			\;
			N^u_2(\mu_u)=\begin{bmatrix} -\frac{\sqrt{6}}{3} & 0 & 0\\[0.5ex] \frac{\sqrt{6}}{6} & 0 & \frac{\sqrt{2}}{2}\\[0.5ex] 0 & \frac{\sqrt{5}}{5} & 0 \end{bmatrix}\!.
		\end{align*}
	}%
	As both, $N^u_1(\lambda_u)$ and $N^u_2(\mu_u)$, span $\R^3$ their intersection still spans $\R^3$.
	Since the intersection of the controllable eigenvalues lacks the first component, we need to pair the uncontrollable eigenvalues $\lambda_u$ and $\mu_u$ and assign an eigenvector with non-zero first component.
\end{ex}

\section{Conclusion}
In this contribution we generalize results for switched systems of size at most
three to subsystems of arbitrary order.
Our approach discusses how to rectify the eigenstructure of two subsystems.
To this effect we introduce the notion of feedback rectifiable pairs, which states the existence of a feedback for every subsystem with the property that the closed-loop subsystems have the same eigenvectors.
It turns out that the eigenspaces of the subsystems may only be rectified if the assigned eigenvalues satisfy certain conditions. 
The formulated necessary and sufficient conditions essentially require that the intersection of two subspaces determined by the subsystems has to span the whole space. 
Moreover, we show that this subspace can be described by a polynomial matrix.
We propose a constructive procedure to calculate this polynomial matrix.
With this characterization at hand, the eigenvalues and eigenvectors for the closed-loop subsystems can be chosen and the respective feedback matrices that stabilize the switched system for arbitrary switching are readily obtained. 

The proposed methodology is not restricted to two subsystems.
In fact the generalization of Definition \ref{def:fbrect} and Theorem \ref{thm:stabilizable} is straightforward.
In Proposition \ref{prop:wieWHR} we need to expand the span to the intersection of all subsystems considered and apply Proposition \ref{prop:intersection} and \ref{cor:rref} iteratively. 
Of course the complexity explodes for large numbers of subsystems and the conditions may become very restrictive.

\bibliographystyle{alpha}
\bibliography{sample,stability_tidy}

\end{document}